\documentclass[11pt,a4paper,twoside]{article}
\usepackage{bm, amsmath, amssymb, amsthm} %, amsthm

\def\RR{\mathbb{R}}
\newcommand\tr{\operatorname{trace}}
\newcommand\Div{\operatorname{div}}

\def\Ric{\operatorname{Ric}}
\def\vol{\operatorname{vol}}

\newtheorem{corollary}{Corollary}
\newtheorem{definition}{Definition}
\newtheorem{example}{Example}
\newtheorem{remark}{Remark}

\newtheorem{proposition}{Proposition}
\newtheorem{theorem}{Theorem}

\author{Vladimir Rovenski\footnote{Department of Mathematics, University of Haifa, Mount Carmel, 31905 Haifa,  Israel
       \newline e-mail: {\tt vrovenski@univ.haifa.ac.il}
       }
        \ and \
        Pawe\l\ Walczak\footnote{Katedra Geometrii,
        Uniwersytet \L\'{o}dzki, ul. Banacha 22,
             90-238  \L\'{o}d\'{z}, Poland
        \newline e-mail: {\tt pawel.walczak@wmii.uni.lodz.pl}}
}

\title{On isometric immersions of almost $k$-product manifolds}

\begin{document}

\date{}

\maketitle

%\tableofcontents

\begin{abstract}
A Riemannian manifold endowed with $k\ge2$ complementary pairwise orthogonal distributions
is called a Riemannian almost $k$-product manifold.
In the article, we study the following problem: {find a relationship between intrinsic and extrinsic invariants of a
Riemannian almost $k$-product manifold isometrically immersed in another Riemannian manifold}.
For such immersions, we establish an inequality that includes the mixed scalar curvature and the square of the mean curvature.
 Although Riemannian curvature tensor belongs to intrinsic geometry, a special part called the mixed curvature is also related to the extrinsic geometry of a Riemannian almost $k$-product manifold.
 Our inequality also contains mixed scalar curvature type invariants related to B.-Y~Chen's $\delta$-invariants.
Applications are given for isometric immersions of multiply twisted and warped products
(we improve some known
inequalities by replacing the sectional curvature with our invariant)
and to problems of non-immersion and non-existence of compact leaves of foliated submanifolds.


\vskip1.5mm
\noindent
\textbf{Mathematics Subject Classifications (2010)}
%Primary
53C12; 53C15; 53C42
%Secondary %53C21

\end{abstract}

\section{Introduction}

%According to the celebrated embedding theorem of J.F.~Nash (1956), every Riemannian manifold can be isometrically embedded in some Euclidean space with %sufficiently high codimension.
A fundamental problem (based on the famous embedding theorem of J. F.~Nash, \cite{na-1}) is to find a simple connection between the  intrinsic
and extrinsic invariants of a submanifold.
 The problem is particularly interesting
% more complicated
 if we assume additional structures, for example, web of foliations, almost $k$-product structure or multiply warped product structure, on a Riemannian submanifold.
%The study of

Multiply warped products have been studied from extrinsic geometry point of view
%was initiated around the beginning of this century
%by B.Y.~Chen andwith several geometers
in a series of articles, see \cite{chen-b,chen-b-v} for an~overview.
%Motivated by the question:  ``What can we conclude from an arbitrary isometric immersion of a warped product in a Riemannian manifold
%with arbitrary codimension?"
B.Y.~Chen and F.~Dillen \cite{chen-d} proved an optimal inequality
(which involves the Laplacian of the warping function and the squared mean curvature)
for warped product manifolds isometrically immersed in arbitrary Riemannian manifolds,
in particular, in Riemannian spaces of constant sectional curvature.
Corresponding inequalities for isometrically immersed doubly warped product manifolds have been obtained by A.~Olteanu in~\cite{ol-1}.
A~sufficient condition for an isometric immersion of a warped product in a Euclidean space to split into a product immersion
was given by S.\,N\"{o}lker \cite{no}.
%and then by several geometers for
%The study of warped product submanifolds became an active research subject, see survey \cite{chen-b}.

%Distributions (i.e., subbundles of the tangent bundle) on a smooth manifold appear in various situations: as nullity spaces of tensors,
%tangent planes of foliations on warped products, etc.
In this article, we consider a Riemannian manifold $(M,g)$ endowed with an \textit{almost $k$-product structure}
(more general than a multiply warped product), i.e.,
$k\ge2$ complementary pairwise orthogonal distributions ${\cal D}_1,\ldots,{\cal D}_k$ (subbundles of the tangent bundle $TM$),
%or foliations,
e.g., \cite{rov-IF-k}.
Such $(M,g;{\cal D}_1,\ldots,{\cal D}_k)$ is called a {Riemannian almost $k$-product manifold}.
Its important example is a multiply twisted or warped product~manifold.
 One can ask when $(M,g,{\cal D}_1,\ldots,{\cal D}_k)$ splits
%locally
into the product of $k$ manifolds.
The best known answer to this question is the de Rham's Decomposition theorem, for example, \cite{MRS-99}:
if each
%distribution
${\cal D}_i$ is parallel (with respect to the Levi-Civita connection), then $M$ \textit{splits}, i.e.,
each point $x\in M$ has a neighborhood $U$,
which
%is a product $F_1\times\ldots\times F_k$ of Riemannian manifolds such that the submanifolds, which
is isometric to the direct product of the maximal integral manifolds for the distributions
obtained by parallel translation of $\,{\cal D}_i(x)$ over $U$.
%are parallel to the factor $F_i$, are integral manifolds of the distribution $\,{\cal D}_i|_U$.
If $M$ is complete and simply connected, the statement is true with~$U=M$.

We study the following problem (more general than \cite[Problem~2]{chen2}):
\textit{find a relationship between intrinsic and extrinsic curvature invariants of
$(M,g;{\cal D}_1,\ldots,{\cal D}_k)$ isometrically immersed in another Riemannian manifold $(\bar M,\bar g)$}.
For such immersions, we prove an inequality that includes the mixed scalar curvature of
$(M,g;{\cal D}_1,\ldots,{\cal D}_k)$ and the square of the mean curvature of the immersion,
and give applications for multiply twisted products and to problems non-immersion and absence of compact leaves of
foliated submanifolds.
 Although Riemannian curvature tensor belongs to intrinsic geometry, a special part called the \textit{mixed curvature} is also related to the extrinsic geometry of $(M,g;{\cal D}_1,\ldots,{\cal D}_k)$,
% almost $k$-product structure,
 see \cite{Rov-Wa-2021}.
Our inequality also contains mixed scalar curvature type invariants related to B.-Y~Chen's $\delta$-invariants for  $k\ge2$, for example, \cite{chen-b,chen-b-v}.

%\begin{remark}\rm
% Theorem~\ref{T-k} and
The results of this article can be applied also to so-called \textit{adapted} isometric immersions
$f: (M,g; {\cal D}_1,\ldots, {\cal D}_k) \to (\bar M,\bar g; \bar{\cal D}_1,\ldots, \bar{\cal D}_k)$, i.e.,
$f_*({\cal D}_i)\subset \bar{\cal D}_{i\,| f(M)}$ for $i=1,\ldots,k$.
Obviously, there are many Riemannian almost $k$-product manifolds (already with integrable distributions, e.g., a $k$-dimensional torus with irrational winding), which cannot be isometrically immersed adapted way
into some Euclidean space endowed with the standard structure of the product of $k$ Euclidean spaces.
Therefore, we supplement the above problem with the following: \textit{charac\-terize almost $k$-product manifolds $(M, g; {\cal D}_1, \ldots, {\cal D}_k)$ admitting  adapted isometric immersions in $\RR^N=\RR^{N_1}\times\ldots\times\RR^{N_k}$ with a Euclidean metric for some $N$
and $N_i$ satisfying $N = \sum_i N_i$}.

The Riemannian almost $k$-product structure on an open simply connected manifold $M$ can be obtained by a special
immersion in $(\bar M,\bar g; \bar{\cal D}_1,\ldots, \bar{\cal D}_k)$.
Let $f_*(TM)$ intersects transversally with every distribution $\bar{\cal D}_i$ restricted to $f(M)$. Then
$f: M \to \bar M$ induces pairwise orthogonal distributions ${\cal D}_i$ on $M$ with induced metric $g$,
i.e., $f_*({\cal D}_i)\subset \bar{\cal D}_{i\,| f(M)}$ and $TM=\bigoplus_{i=1}^{\,k} {\cal D}_i$.
There are topological obstacles to the existence of an almost $k$-product structure on a smooth closed manifold $M$.
For example, if $n_i=1$ for some $i$, then the Euler characteristic $\chi(M)$ is zero;
and if $n_i=1$ for all $i$ and $M$ is an odd-dimensional sphere (thus, $\chi(M)=0$), then maximal $k$ is the (known in topology) number of linear independent vector fields on an $n$-dimensional sphere.
We~do not consider here (no less complicated) topological obstacles for a closed manifold $M^n$ with $n = \sum_i n_i$ to admit a special immersion in $\RR^N=\RR^{N_1}\times\ldots\times\RR^{N_k}$ for some $N = \sum_i N_i$ such that $f_*(TM)$ has an $n_i$-dimensional transversal intersection with every distribution $\RR^{N_i}$ on $\RR^N$.

\smallskip

The article is organized as follows.
Section~\ref{sec:01} contains necessary definitions (e.g., this of the mixed scalar curvature) for a Riemannian almost $k$-product structure.
Section~\ref{sec:02} contains the main theorem providing an
%optimal
inequality (which involves the mixed scalar curvature and the squared mean curvature)
for isometric immersions of such manifolds,
%, in particular, multiply twisted products
and several corollaries with applications to problems of non-immersing and non-existence of compact leaves of foliated submanifolds.
Section~\ref{sec:03} contains applications to isometric immersions of multiply twisted and warped products.

\section{Preliminaries}
\label{sec:01}

Let an $n$-dimensional Riemannian manifold $(M,g)$ with the Levi-Civita connection $\nabla$ and the curvature tensor $R$
be endowed with $k\ge2$ pairwise orthogonal $n_i$-dimensional distributions ${\cal D}_i\ (1\le i\le k)$
(subbundles of the tangent bundle $TM$) of ranks $n_i$ such that $\sum_{\,i=1}^k n_i = n$.
Let ${\cal D}^\bot_i$ be the orthogonal complement to ${\cal D}_i$ in $TM$ and $n^\bot_i=n-n_i$.

Let $P^\bot_i:TM\to{\cal D}^\bot_i$ be the orthoprojector. The {2nd fundamental form} ${h}_i:{\cal D}_i\times {\cal D}_i\to{\cal D}^\bot_i$ and
the {integrability tensor} ${T}_i:{\cal D}_i\times {\cal D}_i\to{\cal D}^\bot_i$ of ${\cal D}_i$ (and similarly, $h^\bot_i$ and $T^\bot_i$) are defined by
\begin{equation}\label{E-def-bT-Wol}
 2\,h_i(X,Y)= P^\bot_i(\nabla_X Y+\nabla_Y X),\quad 2\,T_i(X,Y) = P^\bot_i(\nabla_X Y-\nabla_Y X) = P^\bot_i([X, Y]) .
\end{equation}
Then $H_i=\tr_g h_i$ is called the mean curvature vector field of the distribution ${\cal D}_i$.
 A distribution ${\cal D}_i$ is integrable (or involutive) if $T_i=0$,
and ${\cal D}_i$ is \textit{totally umbilical}, \textit{minimal}, or \textit{totally geodesic},
if  ${h}_i=({H}_i/n_i)\,g,\ {H}_i=0$, or ${h}_i=0$, respectively.

% There exists on $M$ a~local adapted orthonormal frame $\{e_1,\ldots,e_n\}$ on $(M,g;{\cal D}_1,\ldots,{\cal D}_k)$, where
%%\[
% $\{e_1,\ldots, e_{n_1}\}\subset{{\cal D}_1},\
% \{e_{n_{1}+1},\ldots, e_{n_2}\}\subset{{\cal D}_2},\ \ldots \
% \{e_{n_{k-1}+1},\ldots, e_{n_k}\}\subset{{\cal D}_k}$.
%%\]
A plane that nontrivially intersects two of the available $k$ distributions on $(M,g)$
and its sectional curvature are called mixed.
Any unit vectors $X\in{\cal D}_i(x)$ and $Y\in{\cal D}_j(x)$ for $x\in M$ and $i\ne j$ define a mixed plane, and
$K(X,Y)=g(R(X,Y)\,Y,\,X)$ is its \textit{mixed sectional curvature}.

The \textit{mixed scalar curvature} of $(M,g)$ equipped with two complementary orthogonal distributions $({\cal D}_1,{\cal D}_2)$ is an averaged mixed sectional curvature (of planes that intersect ${\cal D}_1$ and ${\cal D}_2={\cal D}^\bot_1$):
\begin{equation*}
%\label{eq-wal2-1}
 {\rm S}_{\,\rm mix}({\cal D}_1,{\cal D}_2)=\sum\nolimits_{\,i\le n_1,\,j>n_1}K(e_i,{e}_j).
\end{equation*}
Here,
%$x\in M$ and
$\{e_i\}$
%on $(M,g;{\cal D}_1,{\cal D}_2)$
is an adapted local orthonormal frame on $TM={\cal D}_1\oplus{\cal D}_2$, i.e.,
 $\{e_1,\ldots, e_{n_1}\}\subset{{\cal D}_1}$.
% and $\{e_{n_{1}+1},\ldots, e_{n_2}\}\subset{{\cal D}_2}$.
If ${\cal D}_1$ or ${{\cal D}_2}$ is spanned by a unit vector $N$ then ${\rm S}_{\,\rm mix}({\cal D}_1,{\cal D}_2)=\Ric_{\,N,N}$.
This concept can be extended for $k>2$ distributions as follows, see \cite{rov-IF-k}.

\begin{definition}\label{D-Smix-k}\rm
%Let ${\cal D}_1,\ldots,{\cal D}_k$ be $k\ge2$ pairwise orthogonal distributions on $(M,g)$.
Let $x\in M$ and $\{e_i\}$ on $(M,g;{\cal D}_1,\ldots,{\cal D}_k)$
be an adapted orthonormal frame on $T_xM$, i.e.,
%\[
 $\{e_1,\ldots, e_{n_1}\}\subset{{\cal D}_1(x)},\
 %\{e_{n_{1}+1},\ldots, e_{n_2}\}\subset{{\cal D}_2(x)},\
 \ldots \
 \{e_{n_{k-1}+1},\ldots, e_{n_k}\}\subset{{\cal D}_k(x)}$.
%\]
%We call
The \textit{mixed scalar curvature} of $(M,g; {\cal D}_1,\ldots, {\cal D}_k)$
%the pair $({\cal D}_i,{\cal D}_j)$
is an averaged mixed sectional curvature, the function on $M$ defined by
\[
 {\rm S}_{\,\rm mix}({\cal D}_1,\ldots,{\cal D}_k)=\sum\nolimits_{\,i<j}\,{\rm S}_{\,\rm mix}({\cal D}_i,{\cal D}_j) .
\]
Here, ${\rm S}_{\,\rm mix}({\cal D}_i,{\cal D}_j)$ is the \textit{mixed scalar curvature} of the pair $({\cal D}_i,{\cal D}_j)$
given at $x\in M$ by
\[
 {\rm S}_{\,\rm mix}({\cal D}_i(x),{\cal D}_j(x)) = \sum\nolimits_{\,n_{i-1}<a\,\le n_i,\ n_{j-1}<b\le n_j} K(e_a,{e}_b),
 \quad i\ne j,
\]
and it does not depend on the choice of~frames.
% see~\cite{rov-IF-k}.
\end{definition}

By Definition~\ref{D-Smix-k}, we get
%the decomposition formula
\begin{equation}\label{E-Dk-Smix}
  2\,{\rm S}_{\,\rm mix}({\cal D}_1,\ldots,{\cal D}_k) = \sum\nolimits_{\,i=1}^k\,
  {\rm S}_{\,\rm mix}({\cal D}_i,{\cal D}^\bot_i),
\end{equation}
where the coefficient 2 is due to the fact that each ${\rm S}_{\,\rm mix}({\cal D}_i,{\cal D}_j)$ participates twice in the sum.
For example, if $k=3$, then
${\rm S}_{\,\rm mix}({\cal D}_1,{\cal D}^\bot_1)={\rm S}_{\,\rm mix}({\cal D}_1,{\cal D}_2)+{\rm S}_{\,\rm mix}({\cal D}_1,{\cal D}_3)$,
etc., hence,
\begin{equation*}
  {\rm S}_{\,\rm mix}({\cal D}_1,{\cal D}^\bot_1)
  {+}{\rm S}_{\,\rm mix}({\cal D}_2,{\cal D}^\bot_2)
  {+}{\rm S}_{\,\rm mix}({\cal D}_3,{\cal D}^\bot_3)
  =2\big(
  {\rm S}_{\,\rm mix}({\cal D}_1,{\cal D}_2){+}{\rm S}_{\,\rm mix}({\cal D}_1,{\cal D}_3){+}{\rm S}_{\,\rm mix}({\cal D}_2,{\cal D}_3)
  \big).
\end{equation*}
 The scalar curvature $\tau=\tr_g\Ric$ (the trace of the Ricci tensor) of $(M,g)$ can be represented~as
\begin{equation}\label{E-Smix-3}
 \tau = 2\,{\rm S}_{\,\rm mix}({\cal D}_1,\ldots,{\cal D}_k) +\sum\nolimits_{\,i=1}^k \tau(\,{\cal D}_i),
\end{equation}
where $\tau(\,{\cal D}_i)$ are scalar curvatures  (functions on $M$) of distributions ${\cal D}_i$.
Thus, if all distributions ${\cal D}_i$ are one-dimensional then $2\,{\rm S}_{\,\rm mix}({\cal D}_1,\ldots,{\cal D}_k)=\tau$.
%-- the scalar curvature.

The {divergence} of a vector~field $X$ on $(M,g)$ is defined using a local orthonormal basis $\{e_i\}$~by
\[
 \Div X =\operatorname{trace}(Y\to\nabla_{Y} X) =\sum\nolimits_{\,i} g(\nabla_{e_i} X, \,e_i).
\]

%\begin{example}[$k=2$]\rm
The~following formula~\cite{wa1}  has many applications:
\begin{equation}\label{E-PW}
%\label{eq-wal_0}
 \Div(H_1 + H^\bot_1) = {\rm S}_{\,\rm mix}({\cal D}_1,{\cal D}^\bot_1) +\|h_1\|^2+\|h^\bot_1\|^2-\|H_1\|^2-\|H^\bot_1\|^2-\|T_1\|^2-\|T^\bot_1\|^2;
\end{equation}
%If both distributions define totally umbilical foliations, then \eqref{E-PW} reads~as
%\begin{equation*}
%\label{E-umb-T6}
% {\rm S}_{\,\rm mix}({\cal D}_1,{\cal D}^\bot_1) = \Div(H_1 + H^\bot_1) +\frac{n_1-1}{n_1}\,\|H_1\|^2 %+\frac{n^\bot_1-1}{n^\bot_1}\,\|H^\bot_1\|^2.
%\end{equation*}
%\end{example}
from \eqref{E-Dk-Smix} and \eqref{E-PW}, for $(M,g;{\cal D}_1,\ldots,{\cal D}_k)$ we obtain, see \cite{rov-IF-k,Rov-Wa-2021},
\begin{eqnarray}\label{E-PW3-k}
\nonumber
 && \Div\sum\nolimits_{\,i=1}^k\big(H_i + H_{i}^\bot\big) = 2\,{\rm S}_{\,\rm mix}({\cal D}_1,\ldots,{\cal D}_k)\\
 && +  \sum\nolimits_{\,i=1}^k\big(\|h_i\|^2 -\|H_i\|^2 - \|T_i\|^2 + \|h_{i}^\bot\|^2 - \|H_{i}^\bot\|^2 -\|T_{i}^\bot\|^2 \big) .
\end{eqnarray}
If all distributions ${\cal D}_1,\ldots,{\cal D}_k$ are totally umbilical, then, see \cite{Rov-Wa-2021},
\begin{equation*}
%\label{E-int-k-umb2}
 \|h_{i}\|^2-\|H_{i}\|^2 = -\frac{n_{i}-1}{n_{i}}\,\|H_{i}\|^2,\quad
 \|h^\bot_{i}\|^2-\|H^\bot_{i}\|^2 = -\frac{n^\bot_{i}-1}{n^\bot_{i}}\,\|H^\bot_{i}\|^2;
\end{equation*}
from this and \eqref{E-PW} we get
\begin{eqnarray}\label{E-umb-T6k}
\nonumber
 %\Div \xi
 &&\quad \Div\sum\nolimits_{\,i=1}^k\big(H_i + H_{i}^\bot\big) = 2\,{\rm S}_{\,\rm mix}({\cal D}_1,\ldots,{\cal D}_k) \\
 &&-\sum\nolimits_{\,i=1}^k\Big(\frac{n_{i}-1}{n_{i}}\,\|H_i\|^2
 +\frac{n^\bot_{i}-1}{n^\bot_{i}}\,\|H^\bot_i\|^2+\|T_i\|^2+\|T^\bot_i\|^2\Big).
\end{eqnarray}

\section{Main results}
\label{sec:02}

Let $\bar{h}:TM\times TM\to TM^\bot$ be the second fundamental form of
an isometric immersion $f: (M^n,g; {\cal D}_1,\ldots, {\cal D}_k) \to (\bar M^m,\bar g)$ with
%$m>n$ and
$\dim {\cal D}_i=n_i$ satisfying $\sum\,_{i=1}^{k}\, n_i = n<m$.
We will identify $M$ with its image $f(M)$ and
put a top ``bar" for the objects related to $(\bar M, \bar g)$.
The~{second fundamental forms} $\bar{h}_i:{\cal D}_i\times {\cal D}_i\to TM^\bot$
are defined similarly to \eqref{E-def-bT-Wol}, as restrictions of $\bar{h}$ on ${\cal D}_i$.
The mean curvature vector $\bar{H}=\tr_{\,\bar g}\,\bar{h}$
of $f$ decomposes as $\bar{H}=\sum_{\,i}\bar{H}_i$, where $\bar{H}_i=\tr_{\,\bar g}\,\bar{h}_{i}$
%\ (i=1,\ldots,k)$
is the mean curvature vector of~${\cal D}_i$.
An~isometric immersion $f$ is called \textit{mixed totally geodesic} if
\[
 \bar{h}(X,Y)=0\quad{\rm for\ all}\ \  X\in{\cal D}_i,\ Y\in{\cal D}_j,\ \ i\ne j.
\]
The Gauss equation for our immersion $f$ (a submanifold $M\subset \bar M$) has the following form:
\begin{equation}\label{E-Gauss-class}
  g(\bar R(Y,Z)U,V) = g(R(Y,Z)U,V) + g(\bar{h}(Y,U), \bar{h}(Z,V)) -g(\bar{h}(Z,U), \bar{h}(Y,V)) ,
\end{equation}
where
$U,V,Y,Z\in TM$ and
$\bar R$ and $R$ are the curvature tensors of $(\bar M^m,\bar g)$ and $(M,g)$, respectively.

\begin{definition}\rm
Let $V_i\ (i=1,\ldots, k)$ be pairwise orthogonal subspaces of $T_x\bar M$ at a point $x\in\bar M$
with $\dim V_i=n_i$.
% satisfying $n_1+\ldots+n_k = n$.
%$(n_1,\ldots, n_k)\in S(n,k)$.
% There exists an orthonormal frame $\{e_i\}$ on $(T\bar M)_x$ such that
%\[
% \{e_1,\ldots, e_{n_1}\}\subset{V_1},\quad
% \{e_{n_{1}+1},\ldots, e_{n_2}\}\subset{V_2},\ \ldots \
% \{e_{n_{k-1}+1},\ldots, e_{n_k}\}\subset{V_k}.
%\]
Let $\{e_i\}$ be an orthonormal frame on $T_x \bar M$ such that
%\[
 $\{e_1,\ldots, e_{n_1}\}\subset V_1$,
 \ldots,
 $\{e_{n_{k-1}+1},\ldots, e_{n_k}\}\subset V_k$.
Define
$\bar{\rm S}_{\,\rm mix}(V_1,\ldots,V_k)= \sum\nolimits_{\,i<j}\sum\nolimits_{\,n_{i-1}<a\,\le n_i,\ n_{j-1}<b\le n_j}
 \bar K(e_a,{e}_b)$
(see Definition~\ref{D-Smix-k}),
and observe that these quantities do not depend on the choice of frames.
Set
\[
 \bar\delta_{\rm mix}(n_1,\ldots, n_k) = \sup\{\bar{\rm S}_{\,\rm mix}(V_1,\ldots,V_k): \dim V_i=n_i \}.
\]
\end{definition}

If the sectional curvature of $(\bar M, \bar g)$ satisfies $\bar K\le C$ and $\sum\nolimits_{\,i=1}^k n_i = n$, then
\[
 \bar\delta_{\rm mix}(n_1,\ldots n_k) \le C\sum\nolimits_{\,i<j} n_i\,n_j \le \frac12\,C\,k\,n^2.
\]

\begin{remark}\rm
The $\bar\delta_{\rm mix}$-invariants are related with B.-Y~Chen's $\delta$-invariants, e.g., \cite{chen-b-v}.
Indeed, 
if $k\ge2$ and the sectional curvature $\bar K\ge0$, then $\bar\delta_{\rm mix}(n_1,\ldots, n_k)\le\bar\delta(n_1,\ldots, n_k)$, where for $x\in \bar M$
\[
 2\bar\delta(n_1,\ldots, n_k)(x)=\bar\tau(x)-\inf\,\{\bar\tau({\cal D}_1)(x)+\ldots +\bar\tau({\cal D}_k)(x)\} .
\]
Also, for $\sum\,_{i=1}^{k}\,n_i=\dim\bar M$ and $k\ge2$, from \eqref{E-Smix-3} we get
$\bar\delta_{\rm mix}(n_1,\ldots, n_k)=\bar \delta(n_1,\ldots, n_k)$.
Note that B.-Y.Chen defines the scalar curvature as half of "trace Ricci", i.e., $\tau=\sum_{\,i<j} K(e_i, e_j)$.
\end{remark}

\begin{theorem}\label{T-k}
Let $f: (M,g; {\cal D}_1,\ldots, {\cal D}_k)\to(\bar M,\bar g)$ be an isometric immersion
of an almost $k$-product manifold in a Riemannian manifold.
Then
\begin{equation}\label{E-ineq-k}
 {\rm S}_{\,\rm mix}({\cal D}_1,\ldots,{\cal D}_k) \le \frac{k-1}{2\,k}\,\|\bar{H}\|^2
 +\bar\delta_{\rm mix}(n_1,\ldots,n_k).
\end{equation}
The equality in \eqref{E-ineq-k} holds if and only if $f$ is mixed totally geodesic,
$\bar{H}_1=\ldots=\bar{H}_k$
%=\frac1k\,\bar{H}$,
and $\bar{\rm S}_{\,\rm mix}({\cal D}_1(x)$, $\ldots,{\cal D}_k(x))=\bar\delta_{\rm mix}(n_1,\ldots,n_k)$ at each point $x\in M$.
\end{theorem}

\begin{proof}
Tracing the Gauss equation \eqref{E-Gauss-class} for the immersion $f$ yields the equality
\begin{equation}\label{E-Si}
  \bar\tau_{\,|M} - \tau = \|\bar{h}\|^2 - \|\bar{H}\|^2,
\end{equation}
where $\bar\tau_{\,|M}$ and $\tau$ are the scalar curvatures of $M$ for the curvature tensors $\bar R$ and $R$, respectively.
Tracing the Gauss equation \eqref{E-Gauss-class} on $M$ along ${\cal D}_i$, we get
\begin{equation}\label{E-Sii}
  \bar\tau(\,{\cal D}_i) - \tau(\,{\cal D}_i) =  \|\bar{h}_i\|^2 - \|\bar{H}_i\|^2,
\end{equation}
where $\bar\tau(\,{\cal D}_i)$ and $\tau(\,{\cal D}_i)$ are the scalar curvatures of ${\cal D}_i$
for the curvature tensors $\bar R$ and $R$.

Suppose that $\bar{H}\ne0$ on an open set $U\subset M$ and
extend over $U$ an adapted local orthonormal frame $\{e_1,\ldots,e_n\}$ of $(M,g)$ by $e_{n+1}$ parallel to $\bar{H}$.
Using the algebraic inequality $a_1^2+\ldots+a_k^2\ge \frac1k\,(a_1+\ldots+a_k)^2$
for $a_i=g(\bar{H}_i, e_{n+1})$, we find
\begin{equation}\label{E-Si3}
  \sum\nolimits_{\,i}\|\bar{H}_i\|^2 \ge \sum\nolimits_{\,i} g(\bar{H}_i, e_{n+1})^2
  \ge \frac1k\,\|\bar{H}\|^2,
\end{equation}
and the equality holds if and only if $\bar{H}_1=\ldots=\bar{H}_k$.
%=\frac1k\,\bar{H}$.
The above inequality is trivially satisfied for $\bar{H}=0$, hence it is valid on $M$.
Note also that
\begin{equation}\label{E-Si4}
 \|\bar{h}\|^2 = \sum\nolimits_{\,i}\|\bar{h}_i\|^2 +\sum\nolimits_{\,i<j}\|\bar{h}^{\rm mix}_{ij}\|^2
 \ge \sum\nolimits_{\,i}\|\bar{h}_i\|^2 ,
\end{equation}
where
$\|\bar{h}^{\rm mix}_{ij}\|^2=\sum_{\,e_a\in{\cal D}_i,\, e_b\in{\cal D}_j} \|\bar{h}(e_a,e_b)\|^2 $,
and $\bar{h}^{\rm mix}_{ij}=0\ (\forall\,i<j)$ if and only if $f$ is mixed totally geodesic.
By \eqref{E-Si}, \eqref{E-Sii}, \eqref{E-Si3}, \eqref{E-Si4} and the equalities
\begin{eqnarray*}
 \bar\tau_{\,|M} = 2\,\bar{\rm S}_{\,\rm mix}({\cal D}_1,\ldots,{\cal D}_k) +\sum\nolimits_{\,i}\bar\tau(\,{\cal D}_i),\\
 \tau = 2\,{\rm S}_{\,\rm mix}({\cal D}_1,\ldots,{\cal D}_k) +\sum\nolimits_{\,i}\tau(\,{\cal D}_i),
\end{eqnarray*}
see \eqref{E-Smix-3}, we obtain
\begin{eqnarray*}
%\label{E-3-10aa}
\nonumber
 && 2\,{\rm S}_{\,\rm mix}({\cal D}_1,\ldots,{\cal D}_k) =  2\,\bar{\rm S}_{\,\rm mix}({\cal D}_1,\ldots,{\cal D}_k)
 + \sum\nolimits_{\,i} (\bar\tau(\,{\cal D}_i) - \tau(\,{\cal D}_i)) + \|\bar{H}\|^2 - \|\bar{h}\|^2 \\
\nonumber
 && \le 2\,\bar\delta_{\rm mix}(n_1,\ldots,n_k) - ( \|\bar{h}\|^2 - \sum\nolimits_{\,i}\|\bar{h}_i\|^2 )
  + (\|\bar{H}\|^2 - \sum\nolimits_{\,i}\|\bar{H}_i\|^2 ) \\
 && \le 2\,\bar\delta_{\rm mix}(n_1,\ldots,n_k) + \frac{k-1}k\,\|\bar{H}\|^2 ,
\end{eqnarray*}
(and the equality holds in the second line if and only if
$\bar{\rm S}_{\,\rm mix}({\cal D}_1(x),\ldots,{\cal D}_k(x))=\bar\delta_{\rm mix}(n_1,\ldots,n_k)$ at each point $x\in M$)
that proves \eqref{E-ineq-k}.
\end{proof}

\begin{example}\rm
Consider distributions ${\cal D}_i\ (i=1,\ldots,k)$ on a domain $M$ on a unit sphere $S^n(1)$ in $\RR^{n+1}$.
Using coordinate charts, we can take integrable distributions ${\cal D}_i$,
and $M$ diffeomorphic to the product of $k$ manifolds.
Extend the distributions ${\cal D}_i$ to distributions $\bar{\cal D}_i\ (i=1,\ldots,k)$ on
an open neighborhood of $M$ in
$\RR^{n+1}\smallsetminus\{0\}$ by applying the homothety
and complementing, say, $\bar{\cal D}_1$ with normals to spheres $S^n(r)$ of radius $r>0$.
Thus $\bar\delta_{\rm mix}(n_1,\ldots,n_k)=0$.

1. For $k=2$, suppose that $M\subset S^n(1)$ is locally diffeomorphic to the product $\RR^{n_1} \times \RR^{n_2}$.

Let $n=3$ and $n_1=1, n_2=2$, then $\|\bar{H}\|^2
%=(1+1+1)^2
=9$, ${\rm S}_{\,\rm mix}({\cal D}_1,{\cal D}_2) = 2$.
Thus \eqref{E-ineq-k} reduces to the inequality $2 < 9/4$ (note that $\bar{H}_1=\frac13\,\bar{H} \not = \frac23\,\bar{H} =\bar{H}_2$).

Let $n=4$, $n_1=n_2=2$ and locally $M\subset S^4(1)$ be diffeomorphic to $\RR^2 \times \RR^2$.
Then $\|\bar{H}\|^2
%=(1+1+1+1)^2
=16$, $\bar{H}_1=\bar{H}_2$,
%=\frac12\,\bar{H}$,
${\rm S}_{\,\rm mix}({\cal D}_1,{\cal D}_2) = 4$ and \eqref{E-ineq-k} reduces to the equality $4 = 16/4$.

2. Let $k=3$ and $n_1=n_2=n_3=1$, i.e., we consider three 1-dimensional distributions ${\cal D}_i\ (i=1,2,3)$ on a domain $M\subset S^3(1)\subset\RR^4$. Recall that $\bar\delta_{\rm mix}(n_1,n_2, n_3)=0$.
Then $\|\bar{H}\|^2
%=(1+1+1)^2
=9$,
$\bar{H}_1=\bar{H}_2=\bar{H}_3$,
%=\frac13\,\bar{H}$,
${\rm S}_{\,\rm mix}({\cal D}_1,{\cal D}_2,{\cal D}_3) = 3$,
 and \eqref{E-ineq-k} reduces to the equality $3 = (2/6)\cdot 9$.
\end{example}

\begin{example}[see \cite{chen-d}]\rm
Let $F_1\times_{u} M_2$ be a multiply warped product with constant mixed sectional curvature.
Let $f_1: M_1\to F_1$ be a minimal immersion, and let $v=(v_1,\ldots, v_k)$, where $v_i\ (i>1)$ are the restrictions  of $u_i$ on $F_1$.
Then the following warped product immersion:
\[
 f=(f_1,{\rm id}): M_1\times_{v} M_2 \to F_1\times_{(u_0,u)} M_2,
\]
is a mixed totally geodesic
%warped product
submanifold, which satisfies the condition $\tr\bar{h}_i=0\ (i\le k)$.
Thus, $f$ satisfies the equality in \eqref{E-ineq-k}, i.e., \eqref{E-ineq-k} is sharp.
\end{example}

\begin{remark}\rm
%For $k=2$,
Observe that $\bar\delta_{\rm mix}(1,q)$ is the supremum of the $q$-th Ricci curvature of $(\bar M,\bar g)$,
which interpolates between Ricci curvature and sectional curvature.
Recall that for $q+1$ orthonormal vectors $V=\{E_0,E_1,\ldots,E_{q}\}$ the \textit{$q$-th Ricci curvature}
of $(\bar M,\bar g)$~is
%\[
 $\overline\Ric_{\,q}(V)=\sum\nolimits_{\,i=1}^{q}\,\bar K(E_0,E_i)$,
%\]
see, for example,~\cite{r-98}.
%Set $\bar r_q= \sup\,\overline\Ric_{\,q}$.
%$\bar r_q= \sup\,\{\overline\Ric_{\,q}(V): V\subset T\bar M,\ \dim V=q\}$.
If $\overline\Ric_{\,q}$
%the $q$-th Ricci curvature of $\bar M$
is constant for some $1<q<\dim \bar M-1$,
then $\bar M$ has constant sectional curvature.
Nevertheless, the class of Riemannian manifolds with positive (or negative) $q$-th Ricci curvature is larger than the class of manifolds with positive (or negative) sectional curvature. For example,
the product of two unit $(n>2)$-dimensional spheres
%$\bar M^{2n}=S^n(1)\times S^n(1),\ (n>2)$
has positive
$(n+1)-$th Ricci curvature and nonnegative sectional curvature.
%Also, compact locally symmetric spaces of rank $\ge 2$ have positive $q(>1)$-th Ricci curvature.
\end{remark}

\begin{corollary}
%\label{T-k2}
Let $f: (M,g; {\cal D}_1,{\cal D}_2) \to (\bar M,\bar g)$ be an isometric immersion
of an almost product manifold
%with $n_1=1$ and $n_2=q$
in a Riemannian manifold.
If ${\cal D}_1$ is spanned by a unit vector field $N$, then
\begin{equation}\label{E-ineq-k2}
 %{\rm S}_{\,\rm mix}({\cal D}_1,{\cal D}_2)
 \Ric_{N,N} \le \frac{1}{4}\,\|\bar{H}\|^2 + \bar r_{q},
% +\bar\delta_{\rm mix}(1,n_2).
\end{equation}
where $q=\dim{\cal D}_2$
and $\bar r_q$ is the supremum of the $q$-th Ricci curvature of $(\bar M,\bar g)$.
The equality in \eqref{E-ineq-k2} holds if and only if $f$ is mixed totally geodesic,
$\bar{H}_1=\bar{H}_2$
%=\frac12\,\bar{H}$
and $\Ric_{N,N}
 %bar S({\cal D}_1(x),{\cal D}_2(x))
 =\bar r_{q}$ at each point~$x\in M$.
\end{corollary}

Using the proof of Theorem~\ref{T-k}, we get the following general
%optimal
inequality.

\begin{corollary}
%\label{T-d}
Let~$f: (M,g) \to (\bar M,\bar g)$ be an isometric immersion of an $n$-dimensional Riemannian manifold in another Riemannian manifold.
Then for any natural numbers $n_1,\ldots,n_k$ such that $\sum_{\,i}n_i=n$ we~get
\begin{equation}\label{E-ineq-dd}
 \delta_{\rm mix}(n_1,\ldots,n_k) \le \frac{k-1}{2\,k}\,\|\bar{H}\|^2 +\bar\delta_{\rm mix}(n_1,\ldots,n_k),
\end{equation}
where $\delta_{\rm mix}(n_1,\ldots,n_k)$ are defined for $(M,g)$ similarly to $\bar\delta_{\rm mix}(n_1,\ldots,n_k)$.
\end{corollary}

\begin{remark}\rm
An ``ideal submanifold" is one that has the least integrated square of the mean curvature (total tension from its ambient space).
Several authors studied $\delta$-ideal submanifolds, e.g., \cite{chen-b-v}.
%$\delta$-Casorati in special Riemannian space forms and R.
Similarly, an isometric immersion of a Riemannian $n$-manifold in another Riemannian manifold is called here
$\delta_{\rm mix}(n_1,\ldots, n_k)$-\textit{ideal} if it satisfies the equality case of \eqref{E-ineq-dd}.
We would like to understand how $\delta_{\rm mix}$-ideal and Chen's $\delta$-ideal submanifolds are related.
\end{remark}

The following consequence of Theorem~\ref{T-k} concerns adapted immersions between almost $k$-product manifolds.
First,
for $(\bar M,\bar g; \bar{\cal D}_1,\ldots,\bar{\cal D}_k)$
define invariants similar to $\bar\delta_{\rm mix}(n_1,\ldots,n_k)$.
Let $V_i\subset\bar{\cal D}_i(x)$ be subspaces at a point $x\in\bar M$ such that $\dim V_i=n_i\ (i=1,\ldots, k)$.
There exists an orthonormal frame $\{e_i\}$ on $T_x\bar M$ such that
%\[
 $\{e_1,\ldots, e_{n_1}\}\subset{V_1},
 %\{e_{n_{1}+1},\ldots, e_{n_2}\}\subset{V_2},\
 \ldots \
 \{e_{n_{k-1}+1},\ldots, e_{n_k}\}\subset{V_k}$.
%\]
%Similarly to Definition~\ref{D-Smix-k},
Put
\[
 \hat{\rm S}(V_1,\ldots,V_k)=\sum\nolimits_{\,i<j}\,\hat{\rm S}(V_i,V_j),
\quad
 \hat{\rm S}(V_i,V_j) = \sum\nolimits_{\,n_{i-1}<a\,\le n_i,\ n_{j-1}<b\le n_j} \bar K(e_a,{e}_b),
\]
and observe that these quantities do not depend on the choice of frames. Finally, set
\begin{eqnarray*}
 \hat\delta_{\rm mix}(n_1,\ldots, n_k)=\sup\{\hat{\rm S}(V_1,\ldots,V_k): V_i\subset\bar{\cal D}_i(x),\ \dim V_i=n_i,\ x\in\bar M \}.
\end{eqnarray*}
Obviously,
$\hat{\rm S}(V_i,V_j)=\bar{\rm S}(V_i,V_j)$, but $\hat\delta_{\rm mix}(n_1,\ldots, n_k)\le \bar\delta_{\rm mix}(n_1,\ldots, n_k)$.

\begin{corollary}
Let $f: (M,g; {\cal D}_1,\ldots,{\cal D}_k) \to (\bar M,\bar g; \bar{\cal D}_1,\ldots,\bar{\cal D}_k)$
be an adapted isometric immersion between Riemannian almost $k$-product manifolds.
Then the following inequality holds:
\begin{equation}\label{E-ineq-2-k}
 {\rm S}_{\,\rm mix}({\cal D}_1,\ldots,{\cal D}_k) \le \frac{k-1}{2\,k}\,\|\bar{H}\|^2
 + \hat\delta_{\rm mix}(n_1,\ldots,n_k) .
\end{equation}
The equality in \eqref{E-ineq-2-k} holds if and only if $f$ is mixed totally geodesic,
$\bar{H}_1=\ldots=\bar{H}_k$ \
%=\frac1k\,\bar{H}$,
and $\hat{\rm S}_{\,\rm mix}({\cal D}_1(x),\ldots,{\cal D}_k(x))=\hat\delta_{\rm mix}(n_1,\ldots,n_k)$ at each point $x\in M$.
\end{corollary}

Applying \eqref{E-PW} to \eqref{E-ineq-k} on a compact manifold $M$, gives the following

\begin{proposition}
%\label{C-new2}
In conditions of Theorem~\ref{T-k}, if $M$ is compact, then
\begin{eqnarray*}
 \int_M\sum\nolimits_{\,i=1}^k(\|H_i\|^2+\|H^\bot_i\|^2-\|h_i\|^2-\|h^\bot_i\|^2+\|T_i\|^2+\|T^\bot_i\|^2)\,d\vol_g \\
 \le \frac{k-1}{2\,k}\int_M \|\bar{H}\|^2\,d\vol_g + \bar\delta_{\rm mix}(n_1,\ldots,n_k)\,{\rm Vol}(M,g).
\end{eqnarray*}
\end{proposition}

%Modifying Stokes theorem on a complete open Riemannian manifold gives the following.

%\begin{lemma}[see Proposition~1 in \cite{csc2010}]\label{L-Div-1}
%Let $(M,g)$ be a complete open Riemannian manifold with a vector field $X$
%such that $\Div X\ge0$. If the norm $\|X\|_g\in{\rm L}^1(M,g)$, then $\Div X\equiv0$.
%\end{lemma}

Recall that $(M,g;{\cal D}_1,\ldots,{\cal D}_k)$ \textit{splits} if
all distributions  ${\cal D}_i$ are integrable and
$M$ is locally the product $F_1\times\ldots\times F_k$ with canonical foliations tangent to ${\cal D}_i$.
%Recall that if a simply connected manifold splits then it is the product.

The following statement is about splitting and non-existence of isometric immersions.

\begin{corollary}\label{C-new2}
In conditions of Theorem~\ref{T-k}, let $(M,g)$ be complete open and ${\cal D}_i\ (i=1,\ldots,k)$ totally umbilical
with $\|H_i\|_g\in{\rm L}^1(M,g)$.
If $\bar\delta_{\rm mix}(n_1,\ldots,n_k)\le0$ $($for example, $(\bar M,\bar g)$ has nonpositive sectional curvature)
and $f(M)$ is minimal, then $(M,g; {\cal D}_1,\ldots, {\cal D}_k)$ splits along the distributions
and $\bar\delta_{\rm mix}(n_1,\ldots,n_k)=0$.
In particular, there are no
%adapted
minimal isometric immersions of $M$ into $\bar M$ when $\bar\delta_{\rm mix}(n_1,\ldots,n_k)<0$.
\end{corollary}

\begin{proof}
Set $\xi= \sum\nolimits_{\,i} \big(H_i + H^\bot_i\big)$.
From \eqref{E-umb-T6k} and \eqref{E-ineq-k}, we get
\begin{equation*}
 \Div \xi \le \frac{k{-}1}{2\,k}\,\|\bar{H}\|^2 + \bar\delta_{\rm mix}(n_1,\ldots,n_k)
  -\sum\nolimits_{\,i=1}^k\Big(\frac{n_{i}-1}{n_{i}}\,\|H_i\|^2
 +\frac{n^\bot_{i}-1}{n^\bot_{i}}\,\|H^\bot_i\|^2+\|T_i\|^2+\|T^\bot_i\|^2\Big).
\end{equation*}
%Modifying Stokes theorem on a complete open Riemannian manifold gives the following.
%\begin{lemma}[see Proposition~1 in \cite{csc2010}]\label{L-Div-1}
Recall \cite[Proposition~1]{csc2010} that if $(M,g)$ is a complete open Riemannian manifold with a vector field $X$
such that $\Div X\ge0$ and the norm $\|X\|_g\in{\rm L}^1(M,g)$, then $\Div X\equiv0$.
%\end{lemma}
Applying this
%Lemma~\ref{L-Div-1}
and using $\bar{H}=0$
rules out the case $\bar\delta_{\rm mix}(n_1,\ldots,n_k)<0$, and in the case of $\bar\delta_{\rm mix}(n_1,\ldots,n_k)\le0$ provides $H_i=0$ and $T_i=0$,
that is $\bar{\cal D}_i$ span totally geodesic foliations of $(\bar M,\bar g)$ and $\bar\delta_{\rm mix}(n_1,\ldots,n_k)\equiv0$.
By De~Rham Decomposition Theorem, $(M,g)$ splits along the distributions ${\cal D}_1,\ldots, {\cal D}_k$.
\end{proof}

\begin{remark}\rm
In Corollary~\ref{C-new2}, if $M$ is closed, then there are no
%adapted
isometric immersions $f$ under a~weaker condition
\[
 \frac{k-1}{2\,k}\,\int_M \|\bar{H}\|^2\,d\vol_g\le -\bar\delta_{\rm mix}(n_1,\ldots,n_k)\,{\rm Vol}(M,g).
\]
\end{remark}

The following statement concerns the existence of compact leaves on a foliated submanifold.

\begin{corollary}
%\label{P-002}
Let $f: (M,g; {\cal D}_1,\ldots, {\cal D}_k) \to (\bar M,\bar g)$
%; \bar{\cal D}_1,\ldots \bar{\cal D}_k)$
be an
% adapted
 isometric immersion
and the distribution ${\cal D}_1$ be minimal (for example, ${\cal D}_1$ defines a minimal foliation).
If there exists a compact submanifold $M'$ of $M$ tangent to ${\cal D}_1$, then at some point of $M'$ we~get
\begin{equation*}
%\label{E-ineq-comp1}
 \frac{k-1}{2\,k}\,\|\bar{H}\|^2 \ge -\bar\delta_{\rm mix}(n_1,\ldots,n_k)
 -\sum\nolimits_{\,i=1}^k (\|h_i\|^2+\|h^\bot_i\|^2 -\|T^\bot_i\|^2).
\end{equation*}
Consequently, if $\,\frac{k-1}{2\,k}\,\|\bar{H}\|^2< -\bar\delta_{\rm mix}(n_1,\ldots,n_k)
 -\sum\nolimits_{\,i=1}^k (\|h_i\|^2+\|h^\bot_i\|^2 -\|T^\bot_i\|^2)$ on $M$,
then there are no compact submanifolds of $M$ tangent to ${\cal D}_1$.
%${\cal F}_1$ has no compact leaves.
\end{corollary}

\begin{proof}
Let $M'$ be a compact submanifold of $M$ tangent to ${\cal D}_1$.
Using \eqref{E-PW} with $H_1=T_1=0$ and the equality
\begin{equation}\label{E-ineq-comp2}
 \Div_{M'}\, H_i = \Div\,H_i +\|H_i\|^2,\quad
 \Div_{M'} H^\bot_i = \Div\,H^\bot_i +\|H^\bot_i\|^2
\end{equation}
along $M'$, from \eqref{E-PW3-k} we obtain
\[
 \Div_{M'} \sum\nolimits_{\,i=1}^k\big( H_i + H^\bot_i \big) = 2\,{\rm S}_{\,\rm mix}({\cal D}_1,\ldots,{\cal D}_k)
 + \sum\nolimits_{\,i=1}^k\big(\|h_i\|^2 + \|h_{i}^\bot\|^2 -\|T_{i}^\bot\|^2 \big) .
\]
Integrating this on $M'$ and using \eqref{E-ineq-k} yields
\[
 \int_{M'} \Big(\frac{k-1}{2\,k}\,\|\bar{H}\|^2 +\bar\delta_{\rm mix}(n_1,\ldots,n_k)
 +\sum\nolimits_{\,i=1}^k\big(\|h_i\|^2 + \|h_{i}^\bot\|^2 -\|T_{i}^\bot\|^2 \big)\Big)d\vol_{M'} \ge 0,
\]
and both claims follow.
\end{proof}

Next, we consider special immersions of a Riemannian manifold in a Riemannian almost $n$-product mani\-fold.
For a $q$-dimensional subspace $V\subset T_x\bar M$ with $q\ge2$, the scalar curvature of $V$ is defined by
$\bar \tau(V)=\sum_{1\le a,b\le q} \bar K(e_a,e_b)$, where is an orthonormal basis of $V$, e.g., \cite{chen-b}.
Set $\bar\tau_{\,q}=\sup\{\bar\tau(V): V\subset T\bar M,\ \dim V=q\}$.
Observe that $2\,\bar\delta_{\rm mix}(\underbrace{1,\ldots,1}_q)=\bar\tau_{\,q}$.

\begin{corollary}
%\label{T-002}
Let $f: (M,g) \to (\bar M,\bar g; \bar{\cal D}_1,\ldots, {\cal D}_n)$ be an isometric immersion
of an $n$-dimensio\-nal Riemannian manifold in an almost $n$-product manifold.
Suppose that for any $x\in M$ the subspace $f(T_xM)$ intersects
%transversally
non-trivially every subspace $\bar{\cal D}_i(x)\ (1\le i\le n)$.
Then
\begin{equation}\label{E-ineq-k-k}
 \tau \le \frac{n-1}{2\,n}\,\|\bar{H}\|^2 + \bar\tau_{\,n}.
\end{equation}
If the equality in \eqref{E-ineq-k-k} holds then $f$ is totally geodesic.
Consequently, if $\,\tau > \bar\tau_{\,n}$
at some point of $M$, then such ``adapted" isometric immersions do not exist.
\end{corollary}

\begin{proof} By assumptions, $f$ induces one-dimensional distributions ${\cal D}_1,\ldots, {\cal D}_n$ on $(M,g)$, and the scalar curvature is  $\tau=2\,{\rm S}_{\,\rm mix}({\cal D}_1,\ldots,{\cal D}_n)$. Applying \eqref{E-ineq-k} with $n_i=1$ completes the proof.
\end{proof}

%extrinsic $q$-th scalar curvature, see \cite{r16}.

\section{Applications to multiply twisted products immersions}
\label{sec:03}

In this section,
we improve some known results on isometric immersions of multiply warped pro\-ducts by replacing the sectional curvature
in the inequality with our invariant $\bar\delta_{\rm mix}(n_1,\ldots, n_k)$.
Certainly, the following results can be extended for multiply doubly twisted product submanifolds.

\begin{definition}[e.g., \cite{chen-2017,chen-d, Rov-Wa-2021}]\rm
Let $M_2=F_2\times\ldots\times F_k$ be the product of $k-1$ manifolds $F_i$, and let $u=(u_2,\ldots,u_k)$.
A~\textit{multiply twisted product} $F_1\times_u M_2$ is the product manifold $F_1\times M_2$ with the metric
$g=g_{F_1}\oplus u_2^2\,g_{F_2}\oplus\ldots\oplus u_k^2\,g_{F_k}$,
where $u_i:F_1\times F_i\to(0,\infty)$ for each $i\ge2$.
Its special cases are twisted products ($k=2$) and {multiply warped product}, i.e., $u_i:F_1\to(0,\infty)$.
% This can be extended as a \textit{multiply doubly twisted product} $F_1\times_{(u_1,u)} M_2$ with $u_1: F_1\times M_2\to(0,\infty)$,
%its {leaves} $F_1\times\{y\}$ and {fibers} corresponding to $F_i\ (i\ge2)$ are totally umbilical submanifolds.
\end{definition}

For a multiply twisted product, the leaves $F_1\times\{y\}$ are totally geodesic,
and the {fibers} corresponding to $F_i\ (i\ge2)$ are totally umbilical,
% submanifolds,
%For multiply twisted products
and we~get $H_i=-n_i P_1\nabla(\log u_i)$ and
\begin{equation}\label{E-Dk-twisted-k}
 {\rm S}_{\,\rm mix}({\cal D}_1,\ldots,{\cal D}_k) = \sum\nolimits_{\,i=2}^k n_i\,\frac{\Delta^{(1)} u_i}{u_i}\,,
\end{equation}
see, for example, \cite[Lemma~1.40]{Rov-Wa-2021},
where $\Delta^{(1)}$ is Laplacian on functions along the {leaves}.

Recall that the Laplacian on functions is given by $\Delta f=-\Div(\nabla f)$, or, using a local orthonormal frame $\{e_i\}$,
it is given by, see, for example, \cite{chen-b},
\[
 \Delta f=\sum\nolimits_{\,i} \big((\nabla_{e_i}e_i)f - e_i(e_i f) \big).
\]

\begin{example}[$k=2$]\rm
%For $k=2$ we get
For a {twisted product} $F_1\times_{u_2} F_2$ of Riemannian manifolds $(F_1,g_1)$ and $(F_2, g_2)$
with a positive function $u_2\in{\rm C}^\infty(F_1\times F_2)$, see \cite{pr}, we have
%\[
  $H_2=-n_2 P_1\nabla(\log u_2)$.
%\]
Then, using
\begin{equation*}
 \Div\,{H}_2 = n_2\,(\Delta_1\,u_2)/u_2 -(n_2^2-n_2)\,\|P_{1}\nabla u_2\|^2/u_2^2,
\end{equation*}
where $\Delta^{(1)}$ is the Laplacian on $(F_1,g_{1})$, and \eqref{E-PW}, we find
%\begin{equation}\label{E-Kmix-warped}
 ${\rm S}_{\,\rm mix}({\cal D}_1,{\cal D}_2) = n_2\frac{\Delta^{(1)} u_2}{u_2}$.
%\end{equation}
\end{example}

The following corollary generalizes results on immersed multiply twisted (warped) products in~\cite{chen-d,ol-1},
with sectional curvature of $\bar M$ instead of $\bar\delta_{\rm mix}(n_1,\ldots,n_k)$.

\begin{theorem}\label{C-004}
Let $f$ be an isometric immersion of a multiply
%doubly
twisted product $F_1\times_{u} M_2$ with $M_2=F_2\times\ldots\times F_k$ and $u=(u_2,\ldots, u_k)$
in a Riemannian manifold $(\bar M,\bar g)$.
%; \bar{\cal D}_1,\ldots,\bar{\cal D}_k)$.
Then
\begin{equation}\label{E-ineq-ktwist}
 \sum\nolimits_{\,i=2}^k n_i\,\frac{\Delta^{(1)} u_i}{u_i} \le \frac{k-1}{2\,k}\,\|\bar{H}\|^2 + \bar\delta_{\rm mix}(n_1,\ldots,n_k) .
\end{equation}
The equality in \eqref{E-ineq-ktwist} holds if and only if $f$ is mixed totally geodesic satisfying
$\bar{H}_1=\ldots=\bar{H}_k$,
%=\frac1k\,\bar{H}$,
and $\bar{\rm S}_{\,\rm mix}({\cal D}_1(x),\ldots,{\cal D}_k(x))=\bar\delta_{\rm mix}(n_1,\ldots,n_k)$ at any point $x\in M$.
\end{theorem}

\begin{proof}
Using \eqref{E-ineq-k} and \eqref{E-Dk-twisted-k}, we obtain \eqref{E-ineq-ktwist}. For the case of equality we apply Theorem~\ref{T-k}.
\end{proof}

\begin{remark}\rm
%\label{C-003}
If $f$ in Theorem~\ref{C-004} is an isometric immersion of a multiply twisted product into a
real space form $\RR^m(c)$, then \eqref{E-ineq-ktwist} reduces to
\begin{equation*}
 \sum\nolimits_{\,i=2}^k n_i\,\frac{\Delta^{(1)} u_i}{u_i} \le \frac{k-1}{2\,k}\,\|\bar{H}\|^2
 + c\sum\nolimits_{\,i<j} n_i\,n_j .
\end{equation*}
\end{remark}

The following statement concerns the existence of a compact factor on an immersed multiply twisted product manifold.
Let $P_1:F_1\times M_2 \to F_1$ be the projector.

%%%%%%%%% NEW
\begin{corollary}
Let $f: F_1\times_{u} M_2 \to (\bar M,\bar g)$
%; \bar{\cal D}_1,\ldots, \bar{\cal D}_k)$
be an
%adapted
isometric immersion of a multiply twisted product with $M_2=F_2\times\ldots\times F_k$ and $u=(u_2,\ldots, u_k)$
in a Riemannian manifold with
\[
 \bar\delta_{\rm mix}(n_1,\ldots,n_k) < -\sum\nolimits_{\,i=2}^k n_i\,\|P_1\nabla\log u_i\|^2.
\]
If $F_1$ is compact, then for any $y\in M_2$ at some point of $F_1\times \{y\}$ we get
\begin{equation*}
%\label{E-ineq-comp1b}
 \frac{k-1}{2\,k}\,\|\bar{H}\|^2 \ge -\bar\delta_{\rm mix}(n_1,\ldots,n_k) -\sum\nolimits_{\,i=2}^k n_i\,\|P_1\nabla\log u_i\|^2 .
\end{equation*}
Consequently, if
\[
 \frac{k-1}{2\,k}\,\|\bar{H}\|^2 < -\bar\delta_{\rm mix}(n_1,\ldots,n_k) - \sum\nolimits_{\,i=2}^k n_i\,\|P_1\nabla\log u_i\|^2
\]
is satisfied on $F_1\times M_2$, then $F_1$ is non-compact.
\end{corollary}

\begin{proof}
Using the equality
 $\frac{\Delta^{(1)} u_i}{u_i} = \Delta^{(1)} \log u_i  - \|P_1\nabla\log u_i\|^2$
and \eqref{E-Dk-twisted-k}, we get
\[
 \Delta^{(1)} \sum\nolimits_{\,i=2}^k n_i\log u_i
 = {\rm S}_{\,\rm mix}({\cal D}_1,\ldots,{\cal D}_k) +\sum\nolimits_{\,i=2}^k  n_i\,\|P_1\nabla\log u_i\|^2 .
\]
Using \eqref{E-ineq-k} or \eqref{E-ineq-ktwist}, we obtain
\[
 \Delta^{(1)} \sum\nolimits_{\,i=2}^k n_i\log u_i \le \frac{k-1}{2\,k}\,\|\bar{H}\|^2
 +\sum\nolimits_{\,i=2}^k  n_i\,\|P_1\nabla\log u_i\|^2 +\bar\delta_{\rm mix}(n_1,\ldots,n_k) .
\]
Let $F_1$ be compact and $y\in M_2$. Integrating on a compact leaf $F_1\times \{y\}$ yields
\[
 \int_{F_1\times \{y\}} \Big(\frac{k-1}{2\,k}\,\|\bar{H}\|^2
 +\sum\nolimits_{\,i=2}^k  n_i\,\|P_1\nabla\log u_i\|^2 +\bar\delta_{\rm mix}(n_1,\ldots,n_k)\Big)\,d\,{\vol}_{\,F_1\times \{y\}}\ge0,
\]
hence, both claims follow.
\end{proof}

The following statement concerns the existence of compact leaves of an isometrically immersed twisted product.
%on a foliated submanifold.

\begin{corollary}
%\label{C-05}
Let $f: F_1\times_{u_2} F_2 \to (\bar M,\bar g)$
%; \bar{\cal D}_1, \bar{\cal D}_2)$
be an isometric immersion of a twisted product in a Riemannian
%almost product
manifold with $\bar\delta_{\rm mix}(n_1,n_2)<0$.
If $F_1$ is compact, then
\begin{equation}\label{E-C7}
 \frac14\,\|\bar{H}\|^2 \ge -\bar\delta_{\rm mix}(n_1,n_2)
\end{equation}
at some point of $F_1\times \{y\}$ for any $y\in F_2$.
Consequently, if the inequality
%\[
 $\frac14\,\|\bar{H}\|^2<-\bar\delta_{\rm mix}(n_1,n_2)$
%\]
is satisfied on $F_1\times \{y\}$ for some $y\in F_2$, then $F_1$ is non-compact.
\end{corollary}

\begin{proof}
Let $F_1$ be compact, then $F_1\times \{y\}$ is compact for any $y\in F_2$.
Denote by ${\cal D}_1$ and ${\cal D}_2$ the distributions on $F_1\times F_2$ tangent to the leaves and fibers.
%, respectively.
From \eqref{E-ineq-k} and \eqref{E-ineq-comp2}, we~get
\[
 n_2\,\frac{\Delta^{(1)}\,u_2}{u_2} = {\rm S}_{\,\rm mix}({\cal D}_1,{\cal D}_2) \le \frac14\,\|\bar{H}\|^2 +\bar\delta_{\rm mix}(n_1,n_2).
\]
Using $\int_{F_1\times \{y\}} \Delta^{(1)}\,u_2\,d\,{\vol}_{F_1\times \{y\}}=0$ yields
\[
 \int_{F_1\times \{y\}} \Big(\frac14\,\|\bar{H}\|^2 +\bar\delta_{\rm mix}(n_1,n_2)\Big)u_2\,d\,{\vol}_{F_1\times \{y\}}\ge0,
\]
and the claims follow.
\end{proof}

Note that for $\bar\delta_{\rm mix}(n_1, n_2)\ge0$, the inequality \eqref{E-C7} is satisfied trivially.

\end{document}